\documentclass[12pt, reqno]{amsart}
\usepackage{amsmath, amsthm, amscd, amsfonts, amssymb, graphicx, color}
\usepackage[bookmarksnumbered, colorlinks, plainpages]{hyperref}
\hypersetup{colorlinks=true,linkcolor=red, anchorcolor=green, citecolor=cyan, urlcolor=red, filecolor=magenta, pdftoolbar=true}

\newtheorem{theorem}{Theorem}[section]
\newtheorem{lemma}[theorem]{Lemma}

\theoremstyle{definition}

\theoremstyle{remark}

\numberwithin{equation}{section}

\begin{document}
\setcounter{page}{1}

\begin{center}
{\Large \textbf{A Dunkl generalization of $q$-parametric Sz\'{a}sz-Mirakjan operators
}}

\bigskip

\textbf{M. Mursaleen} and \textbf{Md. Nasiruzzaman}

Department of\ Mathematics, Aligarh Muslim University, Aligarh--202002, India%
\\[0pt]

mursaleenm@gmail.com; nasir3489@gmail.com \\[0pt%
]

\bigskip

\bigskip

\textbf{Abstract}
\end{center}

\parindent=8mm {\footnotesize {In this paper, we construct a linear positive operators
$q$-parametric Sz\'{a}sz Mirakjan operators generated by the
$q$-Dunkl generalization of the exponential function. We obtain
Korovkin's type approximation theorem for these operators and compute convergence of
these operators by using the modulus of continuity. Furthermore, the rate of
convergence of the operators for functions belonging to the Lipschitz class is presented..}}

\bigskip


\parindent=0mm\emph{AMS Subject Classification (2010):} Primary 41A25; 41A36; Secondary 33C45.










\section{Introduction and preliminaries}

In 1912, S.N Bernstein \cite{bbl1} introduced the following sequence of
operators $B_n:C[0,1] \to C[0,1]$ defined for any $n \in \mathbb{N}
$ and for any $f \in C[0,1]$ such as
\begin{equation}\label{s1}
B_n(f;x)=\sum_{k=0}^n \binom{n}{k}x^k(1-x)^{n-%
k} f\left(\frac{k}{n}\right),~~~~~~~x \in [0,1].
\end{equation}

In 1950, for $x \geq 0$, Sz\'{a}sz \cite{bbl4} introduced the operators
\begin{equation}\label{s2}
S_n(f;x)=e^{-nx}\sum_{k=0}^\infty \frac{(nx)^k}{k!}  f\left(\frac{k}{n}\right),~~~~~~~f \in C[0,\infty).
\end{equation}

In the field of approximation theory, the application of $q$-calculus emerged as a new area in the field of approximation
theory from last two decades. $q$-calculus plays an important role in the natural sciences such as mathematics, physics and chemistry.
It has many applications in number theory, orthogonal polynomials and quantum theory, etc. The development of $q$-calculus has led to
the discovery of various modifications of Bernstein polynomials involving $q$-integers. The aim of these generalizations is to provide
an appropriate and powerful tools to application areas such as numerical analysis, computer-aided geometric design and solutions of differential
equations.

The first $q$-analogue of the well-known Bernstein polynomials was introduced by Lupa\c{s} \cite{bbl2} by
applying the idea of $q$-integers in year 1987. In 1997 Phillips considered another $q$-analogue of the
classical Bernstein polynomials \cite{bbl3}. Later on, many authors introduced $q$-generalization of various operators and investigated several
approximation properties. For instance, $q$- Baskakov-Kantorovich operators in \cite{gupta}; $q$-Sz$\acute{a}$sz-Mirakjan operators in \cite%
{kant}; $q$-analogue of Baskakov and Baskakov-Kantorovich operators in \cite{mah1}; $q$-analogue of Bernstein-Kantorovich operators in \cite{radu}; $q$-analogue of Stancu-Beta operators in \cite{aral2, mur1}; $q$-analogue of Sz$\acute{a}$sz-Kantorovich operators in \cite{mah2}; $q$-Bleimann,
Butzer and Hahn operators in \cite{aral1,ersan}; and $q$-analogue of generalized Bernstein-Shurer operators in \cite{mur3}.

\noindent The $q$-integer $[n]_{q}$, the $q$-factorial $[n]_q!$ and the $q$-binomial coefficient are defined by (see \cite{bbl5,bbl6})
\begin{align*}
[n]_q&:=\left\{
     \begin{array}{ll}
       \frac{1-q^{n}}{1-q}, & \hbox{if~} q\in \mathbb{R}^{+}\setminus\{1\} \\
       n, & \hbox{if~} q=1,
     \end{array}
   \right. \mbox{for $n\in \mathbb{N} $~and~$[0]_q=0$},\\
[n]_{q}!&:=
\left\{
  \begin{array}{ll}
    [n]_{q}[n-1]_{q}\cdots[1]_{q}, & \hbox{$n\geq 1$,} \\
    1, & \hbox{$n=0$,}
  \end{array}
\right.\\
\left[
\begin{array}{c}
n \\
k%
\end{array}
\right] _{q}&:=\frac{[n]_{q}!}{[k]_{q}![n-k]_{q}!},
\end{align*}
respectively.

\noindent The $q$-analogue of $(1+x)^n$ is the polynomial
\begin{equation*}
(1+x)^n_q:=\left\{
\begin{array}{ll}
(1+x)(1+qx)\cdots (1+ q^{n-1}x) & \quad
n=1,2,3,\cdots\\
1 & \quad n=0.\\
\end{array}
\right.
\end{equation*}
A $q$-analogue of the common Pochhammer symbol also called a $q$-shifted factorial is defined as
\begin{equation*}
(x;q)_0=1,~(x;q)_n=\prod\limits_{j=0}^{n-1}(1-q^{j}x),~(x;q)_{\infty }=\prod\limits_{j=0}^{\infty }(1-q^{j}x).
\end{equation*}

\noindent The Gauss binomial formula:
\begin{equation*}
(x+a)_q^n=\sum\limits_{k=0}^n\left[
\begin{array}{c}
n \\
k
\end{array}
\right]_qq^{k(k-1)/2}a^kx^{n-k}.
\end{equation*}

The $q-$analogue of Bernstein operators in \cite{bbl3} defined as follows:

\begin{equation}
B_{n,q}(f;x)=\sum\limits_{k=0}^{n}\left[
\begin{array}{c}
n \\
k%
\end{array}%
\right] _{q}x^{k}\prod\limits_{s=0}^{n-k-1}(1-q^{s}x)~~f\left( \frac{%
[k]_{q} }{[m]_{q}}\right) ,~~x\in [0,1], n \in \mathbb{N}.
\end{equation}

There are two $q$-analogue of the exponential function $e^z$, defined as (see also \cite{bbl7})\newline

For $\mid z \mid < \frac{1}{1-q}$ and $ \mid q \mid <1$,
\begin{equation}\label{zfr1}
e(z)= \sum_{k=0}^\infty \frac{z^k}{k!}= \frac{1}{1-\left((1-q)z\right)_q^\infty},
\end{equation}
and for $\mid q \mid<1$,
\begin{equation}\label{zfr2}
E(z)= \prod_{j=0}^\infty \left(1+(1-q)q^jz\right)_q^\infty=\sum_{k=0}^\infty q^{\frac{k(k-1)}{2}}\frac{z^k}{k!}
= \left(1+(1-q)z\right)_q^\infty,
\end{equation}
where $(1-x)_q^\infty=\prod_{j=0}^\infty(1-q^jx)$.


In \cite{bbl8} $q$-Sz\'{a}sz-Mirakjan  operators were defined as follows:

\begin{equation}\label{nt1}
S_{n,q}(f;x):= E\left(-\frac{[n]_qx}{b_n}\right)\sum_{k=0}^\infty f \left(\frac{[k]_qb_n}{[n]_q}\right)
 \frac{[n]_q^kx^k}{[k]_q!b_n^k},
\end{equation}
where $0 \leq x < \frac{b_n}{(1-q)[n]_q},~~~f \in C[0,\infty)$ and $\{b_n\}$ is a sequence of positive numbers such that
$\lim_{n \to \infty}b_n=\infty.$ Since from the structural point of view the operators defined in \eqref{nt1} have the
convergence properties similar to the Bernstein-Chlodowsky operators. Someone refer these operators to
$q$-parametric Sz\'{a}sz-Mirakjan-Chlodowsky operators and define $q$-parametric Sz\'{a}sz-Mirakjan operators \cite{bbl12}:


For $x \geq 0,~~~f \in C[0,\infty),\mu \geq 0,~~~n \in \mathbb{N}$
Sucu \cite{bbl9} defined a Dunkl analogue of Sz\'{a}sz operators via a generalization of the exponential function given by \cite{bbl10} as follows:

\begin{equation}
S_{n}^*(f;x):= \frac{1}{e_\mu(nx)}\sum_{k=0}^\infty \frac{(nx)^k}{\gamma_\mu(k)} f \left(\frac{k+2\mu\theta_k}{n}\right),
\end{equation}
where $$e_\mu(x)= \sum_{n=0}^\infty \frac{x^n}{\gamma_\mu(n)}.$$
Also here
$$\gamma_\mu(2k)= \frac{2^{2k}k!\Gamma\left(k+\mu+\frac{1}{2}\right)}{\Gamma\left(\mu+\frac{1}{2}\right)},$$
and
$$\gamma_\mu(2k+1)= \frac{2^{2k+1}k!\Gamma\left(k+\mu+\frac{3}{2}\right)}{\Gamma\left(\mu+\frac{1}{2}\right)}.$$

There is given a recursion for $\gamma_\mu$
$$\gamma_\mu(k+1)= (k+1+2\mu\theta_{k+1})\gamma_\mu(k),~~~~k =0,1,2,\cdots,$$
where
\[ \theta_k =
  \begin{cases}
   0       & \quad \text{if } k \in 2\mathbb{N} \\
    1  & \quad \text{if } k \in 2\mathbb{N}+1\\
  \end{cases}
\]
Ben Cheikh et al. \cite{bbl11} stated the $q$-Dunkl classical $q$-Hermite type polynomials. They gave definitions of $q$-Dunkl analogues of
exponential functions, recursion relations and notaions for $\mu > \frac{1}{2}$ and $0<q<1$, respectively.
\begin{equation}\label{r1}
e_{\mu,q}(x)= \sum_{n=0}^\infty \frac{x^n}{\gamma_{\mu,q}(n)},~~~x \in [0,\infty)
\end{equation}
\begin{equation}\label{r2}
E_{\mu,q}(x)= \sum_{n=0}^\infty \frac{q^{\frac{n(n-1)}{2}}x^n}{\gamma_{\mu,q}(n)},~~~x \in [0,\infty)
\end{equation}
\begin{equation}\label{r3}
\gamma_{\mu,q}(n+1)= \left(\frac{1-q^{2\mu\theta_{n+1}+n+1}}{1-q}\right)\gamma_{\mu,q}(n),~~~~~n \in \mathbb{N},
\end{equation}

\[ \theta_n =
  \begin{cases}
   0       & \quad \text{if } n \in 2\mathbb{N}, \\
    1  & \quad \text{if } n \in 2\mathbb{N}+1.\\
  \end{cases}
\]
An explicit formula for $\gamma_{\mu,q}(n)$ is
\begin{equation*}
\gamma_{\mu,q}(n)= \frac{(q^{2\mu+1},q^2)_{[\frac{n+1}{2}]}(q^2,q^2)_{[\frac{n}{2}]}}{(1-q)^n}\gamma_{\mu,q}(n),~~~~~n \in \mathbb{N}.
\end{equation*}

And some of the special cases of $\gamma_{\mu,q}(n)$ defined as:

\begin{equation*}
\gamma_{\mu,q}(0)=1, ~~~~~~\gamma_{\mu,q}(1)=\frac{1-q^{2\mu+1}}{1-q},~~~~\gamma_{\mu,q}(2)=\left(\frac{1-q^{2\mu+1}}{1-q}\right)\left(\frac{1-q^2}{1-q}\right),
\end{equation*}

 \begin{equation*}
\gamma_{\mu,q}(3)=
\left(\frac{1-q^{2\mu+1}}{1-q}\right) \left(\frac{1-q^2}{1-q}\right) \left(\frac{1-q^{2\mu+3}}{1-q}\right),
\end{equation*}

\begin{equation*}
\gamma_{\mu,q}(4)=
\left(\frac{1-q^{2\mu+1}}{1-q}\right) \left(\frac{1-q^2}{1-q}\right) \left(\frac{1-q^{2\mu+3}}{1-q}\right) \left(\frac{1-q^4}{1-q}\right).
\end{equation*}

In \cite{bbl13}, G\"{u}rhan I\c{c}\"{o}z gave a Dunkl generalization of Sz\'{a}sz operators via $q$-calculs as:
\begin{equation} \label{ss1}
D_{n,q}(f;x)=\frac{1}{e_{\mu,q}([n]_qx)}\sum_{k=0}^\infty \frac{([n]_qx)^k}{\gamma_{\mu,q}(k)}f \left( \frac{1-q^{2\mu\theta_k+k}}{1-q^n}\right).
\end{equation}



In this paper, we define a Dunkl generalization of $q$-parametric Sz\'{a}sz-Mirakjan operators \cite{bbl12}:\newline

For any $x\in [0,\infty),~~~f \in C[0,\infty) ~~n \in \mathbb{N}, ~~0<q<1,$ and $\mu>\frac{1}{2}$, we define

\begin{equation} \label{ss1}
D_{n,q}^*(f;x)=\frac{1}{E_{\mu,q}([n]_qx)}\sum_{k=0}^\infty \frac{([n]_qx)^k}{\gamma_{\mu,q}(k)}q^{\frac{k(k-1)}{2}}f \left( \frac{1-q^{2\mu\theta_k+k}}{q^{k-2}(1-q^n)}\right).
\end{equation}

Recently, Mursaleen et al. \cite{n1} applied $(p, q)$-calculus in approximation, which is a
generalization of $q$-calculus. They introduced first $(p, q)$-analogue of Bernstein operators.
The $(p, q)$-calculus$, 0 < q < p \leq 1$ in which for $p = 1$, $(p, q)$-integers reduce to $q$-integers.
They have also introduced and studied approximation properties of $(p,q)$-analogue of Bernstein-Schurer operators \cite{n2},
Bernstein-Stancu operators\cite{n3}, Bleimann-Butzer-Hahn operators \cite{n4}, bivariate Bleimann Butzer-Hahn operators \cite{n5}
 and higher order generalization of Bernstein type opertaors \cite{n6}. For details on $q$ and $(p,q)$-calculus refer to \cite{bbl5,n7,n8}.



\section{Main results}

\begin{lemma}\label{lm1}
Let $D_{n,q}^*(.~;~.)$ be the operators given by \eqref{ss1}, then we have the following identities:

\begin{enumerate}
\item \label{lm11} $D_{n,q}^*(e_0;x)=1$

\item\label{lm12} $D_{n,q}^*(e_1;x)=qx$

\item \label{lm13}$qx^2+\frac{q^{2(1+\mu)}}{[n]_q}[1-2\mu]_qx \leq D_{n,q}^*(e_2;x)\leq qx^2+\frac{q^{2(1+\mu)}}{[n]_q}[1+2\mu]_qx$,
\end{enumerate}
where $e_j(t)=t^j,~~j=0,1,2, \cdots.$
\end{lemma}

\begin{lemma}\label{lm2}
Let the operators $D_{n,q}^*(.~;~.)$ be given by \eqref{ss1}, then we have the following identities:

\begin{enumerate}
\item\label{lm21} $D_{n,q}^*(e_1-1;x)=qx-1$

\item\label{lm22} $D_{n,q}^*(e_1-x;x)=(q-1)x $

\item\label{lm23} $(1-q)x^2+ \frac{q^{2(1+\mu)}}{[n]_q}[1-2\mu]_{q}x \leq D_{n,q}^*((e_1-x)^2;x)\leq  (1-q)x^2+ \frac{q^{2(1+\mu)}}{[n]_q}[1+2\mu]_{q}x.$
\end{enumerate}
\end{lemma}

\textbf{\ Korovkin type approximation properties.}\\

We obtain the Korovkin's type approximation properties for
our operators defined by \eqref{ss1}.\newline

Let $C_B(\mathbb{R^+})$ be the set of all bounded and continuous functions on $\mathbb{R^+}=[0,\infty)$, which is linear
normed space with
$$\parallel f \parallel_{C_B}=\sup_{x\geq 0}\mid f(x) \mid.$$
And also let
$$H:=\{f:x \in [0,\infty), \frac{f(x)}{1+x^2}~~~\mbox{is}~~~\mbox{convergent}~~~\mbox{as}~~~x \to \infty\}.$$


\parindent=8mmIn order to obtain the convergence results for the operators $%
D_{n,q}^*(.,.)$, we take $q=q_{n}$ where $q_{n}\in (0,1)$ and satisfying,
\begin{align}\label{nas5}
\lim_{n}q_{n}\to 1,~~~~~~\lim_{n}q_{n}^n \to a
\end{align}

\begin{theorem}\label{th1}
Let $q=q_n$ satisfying \eqref{nas5}, for $0<q_n< 1$ and if $%
D_{n,q_n}^*(.~;~.)$ be the operators given by \eqref{ss1}. Then for any function $f \in X[0,\infty) \cap H$,
\begin{equation*}
 D_{n,q_n}^*(f;x)=f(x)
\end{equation*}
is uniformly on each compact subset of $[0,\infty)$.
\end{theorem}

\begin{proof}
The proof is based on the well known Korovkin's theorem regarding the
convergence of a sequence of linear and positive operators, so it is enough
to prove the conditions
\begin{equation*}
{D}_{n,q_n}^*((e_j;x)=x^j,~~~j=0,1,2,~~~\{\mbox{as}~
n \to \infty\}
\end{equation*}
uniformly on $[0,1]$.\newline
 Clearly from \eqref{nas5} and $\frac{1}{[n]_{q_n}} \to 0,~~~ (n \to \infty)$ we have

\begin{equation*}
\lim_{n \to \infty}{D}_{n,q_n}^*(e_1;x)=x,~~~\lim_{n \to \infty}{D}_{n,q_n}^*(e_2;x)=x^2.
\end{equation*}
Which complete the proof.
\end{proof}

We recall the weighted spaces of the functions on $\mathbb{R}^+$, which are defined as follows:
\begin{eqnarray*}
P_\rho(\mathbb{R}^+) & = & \left\{ f: \mid f(x) \mid \leq M_{f}\rho(x)\right\}, \\
Q_\rho(\mathbb{R}^+) & = & \left\{ f: f \in P_\rho(\mathbb{R}^+) \cap C[0,\infty)\right\}, \\
Q^k_\rho(\mathbb{R}^+) & = & \left\{ f: f \in Q_\rho(\mathbb{R}^+)~~~\mbox{and}~~~ \lim_{x \to \infty}\frac{f(x)}{\rho(x)}
=k(k ~~~\mbox{is}~~~\mbox{a}~~~\mbox{constant})\right\},
\end{eqnarray*}
where $\rho(x)=1+x^2$ is a weight function and $M_f$ is a constant depending only on $f$. And $Q_\rho(\mathbb{R}^+)$ is a normed
space with the norm $\parallel f \parallel_\rho=\sup_{x \geq 0} \frac{\mid f(x) \mid }{\rho(x)}$.

\begin{theorem}\label{th2}
Let $q=q_n$ satisfying \eqref{nas5}, for $0<q_n< 1$ and if $%
D_{n,q_n}^*(.~;~.)$ be the operators given by \eqref{ss1}. Then for any function $f \in Q^k_\rho(\mathbb{R}^+)$ we have
\begin{equation*}
\lim_{n\to \infty} \parallel D_{n,q_n}^*(f;x)-f \parallel_\rho=0.
\end{equation*}
\end{theorem}

\begin{proof}
From Lemma \ref{lm1}, the first condition of \eqref{lm11} is fulfilled for $\tau=0$. Now for $\tau=1,2$
it is easy to see that from (\ref{lm12}), (\ref{lm13}) of Lemma \ref{lm1} by using \eqref{nas5} \newline
\begin{equation*}
\parallel D_{n,q_n}^* \left( e_1) ^{\tau};x\right) -x ^{\tau }\parallel _{\rho} =0.
\end{equation*}%
This complete the proof.

\end{proof}

\textbf{\ Rate of Convergence.}\\ \newline

Here we calculate the rate of convergence of operators \eqref{ss1} by
means of modulus of continuity and Lipschitz type maximal functions.

Let $f \in C[0,\infty] $, and the modulus of continuity of $f$ denoted by $%
\omega(f,\delta)$ gives the maximum oscillation of $f$ in any interval of
length not exceeding $\delta>0$ and it is given by the relation
\begin{equation}\label{sn1}
\omega(f,\delta)=\sup_{\mid y-x \mid \leq \delta} \mid f(y)-f(x) \mid,~~~x,y \in [0,\infty).
\end{equation}
It is known that $\lim_{\delta\to 0+}\omega(f,\delta)=0$ for $f \in C[0,
\infty)$ and for any $\delta >0$ one has
\begin{equation}\label{sn2}
\mid f(y)- f(x) \mid \leq \left( \frac{\mid y-x \mid}{\delta}+1\right)\omega(f,\delta).
\end{equation}

\begin{theorem}
Let $q=q_{n}$ satisfy \eqref{nas5} for $x\geq 0, ~~0<q_{n}< 1$ and if  $%
D_{n,q_n}^*(.~;~.)$  be the operators defined by \eqref{ss1}. Then for
any function $f\in \tilde{C}[0,\infty)$, we have
\begin{equation*}
\mid D_{n,q_n}^*(f;x)-f(x)\mid \leq \left\{ 1+\sqrt{(1-q_n)[n]_{q_n}x^2+q_n^{2(1+\mu)}[1+2\mu]_{q_n}x}\right\}
\omega\left(f;\frac{1}{\sqrt{[n]_{q_n}}}\right),
\end{equation*}%
where
$\tilde{C}[0,\infty)$ is the space of uniformly continuous functions on $\mathbb{R}^+$ and
$\omega(f,\delta)$ is the modulus of continuity of the function $f \in \tilde{C}[0,\infty)$ defined in \eqref{sn1}.
\end{theorem}

\begin{proof}We prove it by using the result \eqref{sn1},\eqref{sn2} and Cauchy-Schwarz inequality:\newline
$\mid D_{n,q}^*(f;x)-f(x) \mid$
\begin{eqnarray*}
&\leq & \frac{1}{E_{\mu,q}([n]_qx)}\sum_{k=0}^\infty \frac{([n]_qx)^{k}}{\gamma_{\mu,q}(k)}q^{\frac{k(k-1)}{2}}
\bigg{|}f\left( \frac{1-q^{2\mu\theta_k+k}}{q^{k-2}(1-q^n)}\right) -f(x) \bigg{|}\\
&\leq & \frac{1}{E_{\mu,q}([n]_qx)}\sum_{k=0}^\infty \frac{([n]_qx)^{k}}{\gamma_{\mu,q}(k)}q^{\frac{k(k-1)}{2}}
\left\{1+ \frac{1}{\delta} \bigg{|}\left( \frac{1-q^{2\mu\theta_k+k}}{q^{k-2}(1-q^n)}\right) -x \bigg{|}\right\}\omega(f;\delta)\\
&= &\left\{1+\frac{1}{\delta}\left( \frac{1}{E_{\mu,q}([n]_qx)}\sum_{k=0}^\infty \frac{([n]_qx)^{k}}{\gamma_{\mu,q}(k)}q^{\frac{k(k-1)}{2}}
\bigg{|} \frac{1-q^{2\mu\theta_k+k}}{q^{k-2}(1-q^n)}-x\bigg{|}\right) \right\}\omega(f;\delta)\\
&\leq &\left\{1+\frac{1}{\delta}\left( \frac{1}{E_{\mu,q}([n]_qx)}\sum_{k=0}^\infty \frac{([n]_qx)^{k}}{\gamma_{\mu,q}(k)}q^{\frac{k(k-1)}{2}}
\left( \frac{1-q^{2\mu\theta_k+k}}{q^{k-2}(1-q^n)}-x\right)^2\right)^{\frac{1}{2}} \left(D_{n,q}^*(e_0;x)\right)^{\frac{1}{2}}\right\}\omega(f;\delta)\\
&= & \left\{1+\frac{1}{\delta} \left(D_{n,q}^*(e_1-x)^2;x\right)^{\frac{1}{2}}\right\}\omega(f;\delta)\\
&\leq & \left\{1+\frac{1}{\delta} \sqrt{(1-q)x^2+\frac{q^{2(1+\mu)}}{[n]_q}[1+2\mu]_qx}\right\}\omega(f;\delta),
\end{eqnarray*}
if we choose $\delta=\delta_n=\sqrt{\frac{1}{[n]_q}}$, then we get our result.
\end{proof}

Now we give the rate of convergence of the operators ${D}%
_{n,q}^*(f;x) $ defined in \eqref{ss1} in terms of the elements of the usual Lipschitz class $%
Lip_{M}(\nu )$.

Let $f\in C[0,\infty)$, $M>0$ and $0<\nu \leq 1$. We recall that $f
$ belongs to the class $Lip_{M}(\nu )$ if
\begin{equation}\label{nn1}
Lip_{M}(\nu )=\left\{ f: \mid f(\zeta_1)-f(\zeta_2)\mid \leq M\mid \zeta_1-\zeta_2\mid^{\nu }~~~(\zeta_1,\zeta_2\in [0,\infty))\right\}
\end{equation}%
is satisfied.\newline

\begin{theorem}\label{sn1}
Let $D_{n,q}^*(.~;~.)$ be the operator defined in \eqref{ss1}.
 Then for each $f\in Lip_{M}(\nu ),~~(M>0,~~~0<\nu \leq 1)$ satisfying \eqref{nn1} we have
\begin{equation*}
\mid D_{n,q}^*(f;x)-f(x)\mid \leq M \left(\lambda_{n}(x)\right)^{\frac{\nu}{2}}
\end{equation*}
where $\lambda_{n}(x)=D_{n,q}^*\left((e_1-x)^2;x\right)$.
\end{theorem}

\begin{proof}
We prove it by using the result \eqref{nn1} and H\"{o}lder inequality:\newline
$$ \mid D_{n,q}^*(f;x)-f(x)\mid \leq \mid D_{n,q}^*(f(e_1)-f(x);x)\mid \leq  D_{n,q}^*\left(\mid f(e_1)-f(x)\mid;x\right)
\leq \mid M D_{n,q}^*\left(\mid e_1-x\mid^\nu;x\right).$$
Therefore\newline

$\mid D_{n,q}^*(f;x)-f(x) \mid$\\
$\leq  M \frac{1}{E_{\mu,q}([n]_qx)}\sum_{k=0}^\infty \frac{([n]_qx)^{k}}{\gamma_{\mu,q}(k)}q^{\frac{k(k-1)}{2}}
\bigg{|} \frac{1-q^{2\mu\theta_k+k}}{q^{k-2}(1-q^n)}-x \bigg{|}^\nu$\\
$\leq M \frac{1}{E_{\mu,q}([n]_qx)}\sum_{k=0}^\infty \left(\frac{([n]_qx)^{k}q^{\frac{k(k-1)}{2}}}{\gamma_{\mu,q}(k)}\right)^{\frac{2-\nu}{2}}
\left(\frac{([n]_qx)^{k}q^{\frac{k(k-1)}{2}}}{\gamma_{\mu,q}(k)}\right)^{\frac{\nu}{2}}
\bigg{|} \frac{1-q^{2\mu\theta_k+k}}{q^{k-2}(1-q^n)}-x \bigg{|}^\nu$\\
$\leq  M  \left(\frac{1}{\left(E_{\mu,q}([n]_qx)\right)}\sum_{k=0}^\infty \frac{([n]_qx)^{k}q^{\frac{k(k-1)}{2}}}{\gamma_{\mu,q}(k)}\right)^{\frac{2-\nu}{2}}
 \left(\frac{1}{\left(E_{\mu,q}([n]_qx)\right)}\sum_{k=0}^\infty \frac{([n]_qx)^{k}q^{\frac{k(k-1)}{2}}}{\gamma_{\mu,q}(k)}\bigg{|} \frac{1-q^{2\mu\theta_k+k}}{q^{k-2}(1-q^n)}-x \bigg{|}^2\right)^{\frac{\nu}{2}}$\\
$ \leq  M \left(D_{n,q}^*(e_1-x)^2;x\right)^{\frac{\nu}{2}}$.\newline
Which complete the proof.

\end{proof}

We have $C_B[0,\infty)$ is the space of all bounded and continuous
functions on $\mathbb{R}^+=[0,\infty)$ and
\begin{equation}\label{t2}
C_B^2(\mathbb{R}^+)=\{ g \in C_B(\mathbb{R}^+):g^{\prime },g^{\prime \prime } \in C_B(\mathbb{R}^+)\},
\end{equation}
with the norm
\begin{equation}\label{t1}
\parallel g \parallel_{C_B^2(\mathbb{R}^+)}=\parallel g \parallel_{C_B(\mathbb{R}^+)}+
\parallel g^{\prime } \parallel_{C_B(\mathbb{R}^+)}+\parallel g^{\prime \prime } \parallel_{C_B(\mathbb{R}^+)},
\end{equation}
also
\begin{equation}\label{t3}
\parallel g \parallel_{C_B(\mathbb{R}^+)}=\sup_{x \in \mathbb{R}^+}\mid g(x) \mid.
\end{equation}

\begin{theorem}\label{sn2}
Let $D_{n,q}^*(.~;~.)$ be the operator defined in \eqref{ss1}. Then for any $g \in C_B^2(\mathbb{R}^+)$ we have
\begin{equation*}
\mid D_{n,q}^*(f;x)-f(x)\mid \leq \left((q-1)x+\frac{\lambda_{n}(x)}{2}\right) \parallel g \parallel_{C_B^2(\mathbb{R}^+)}
\end{equation*}
where $\lambda_{n}(x)$ is given in Theorem \ref{sn1}.
\end{theorem}

\begin{proof}
Let $g \in C_B^2(\mathbb{R}^+)$, then by using the generalized mean value theorem in the Taylor series expansion
we have
$$ g(e_1)=g(x)+g^{\prime}(x)(e_1-x)+g^{\prime \prime}(\psi)\frac{(e_1-x)^2}{2},~~~\psi \in (x,e_1).$$
By applying linearity property on $D_{n,q}^*,$ we have
$$D_{n,q}^*(g,x)-g(x)=g^{\prime}(x)D_{n,q}^*\left((e_1-x);x\right)+\frac{g^{\prime \prime}(\psi)}{2}D_{n,q}^*\left((e_1-x)^2;x\right),$$
which imply that,
\begin{equation*}
\mid D_{n,q}^*(g;x)-g(x)\mid \leq  (q-1)x\parallel g^{\prime} \parallel_{C_B(\mathbb{R}^+)}+
\left( (1-q)x^2+\frac{q^{2(1+\mu)}}{[n]_q}[1+2\mu]_q x \right)
\frac{\parallel g^{\prime \prime} \parallel_{C_B(\mathbb{R}^+)}}{2}
\end{equation*}
From \eqref{t1} we have ~~~$\parallel g^{\prime} \parallel_{C_B[0,\infty)}\leq \parallel g \parallel_{C_B^2[0,\infty)}$.\newline
\begin{equation*}
\mid D_{n,q}^*(g;x)-g(x)\mid \leq  (q-1)x\parallel g \parallel_{C_B^2(\mathbb{R}^+)}+
\left( (1-q)x^2+\frac{q^{2(1+\mu)}}{[n]_q}[1+2\mu]_q x \right)
\frac{\parallel g \parallel_{C_B^2(\mathbb{R}^+)}}{2}.
\end{equation*}

This complete the proof from \ref{lm23} of Lemma \ref{lm2}.
\end{proof}

The Peetre's $K$-functional is defined by
\begin{equation}\label{zr1}
K_{2}(f,\delta )=\inf_{C_B^2(\mathbb{R}^+)} \left\{ \left( \parallel f-g\parallel_{C_B(\mathbb{R}^+)}
+\delta \parallel g^{\prime \prime }\parallel_{C_B^2(\mathbb{R}^+)} \right):g\in \mathcal{W}%
^{2}\right\} ,
\end{equation}%
where
\begin{equation}\label{zr2}
\mathcal{W}^{2}=\left\{ g\in C_B(\mathbb{R}^+):g^{\prime },g^{\prime \prime }\in
C_B(\mathbb{R}^+)\right\} .
\end{equation}%
Then there exits a positive constant $C>0$ such that $K%
_{2}(f,\delta )\leq C \omega _{2}(f,\delta ^{\frac{1}{2}}),~~\delta
>0$, where the second order modulus of continuity is given by
\begin{equation}\label{zr3}
\omega _{2}(f,\delta ^{\frac{1}{2}})=\sup_{0<h<\delta ^{\frac{1}{2}%
}}\sup_{x\in  \mathbb{R}^+}\mid f(x+2h)-2f(x+h)+f(x)\mid .
\end{equation}%

\begin{theorem}\label{sn3}
Let $D_{n,q}^*(.~;~.)$ be the operator defined in \eqref{ss1} and $C_B[0,\infty)$ is the space of all bounded and continuous
functions on $\mathbb{R}^+$. Then for $x\in \mathbb{R}^+,~~~f \in C_B(\mathbb{R}^+)$ we have
\begin{equation*}
\mid D_{n,q}^*(f;x)-f(x)\mid \leq 2M \left\{ \omega_2\left(f; \sqrt{\frac{2x(q-1)+ \lambda_n(x)}{4}}\right)+
\min\left(1,\frac{2x(q-1)+ \lambda_n(x)}{4}\right)\parallel f\parallel_{C_B(\mathbb{R}^+)}\right\},
\end{equation*}
where $M$ is a positive constant, $\lambda_{n}(x)$ is given in Theorem \ref{sn1}
and $\omega_2(f;\delta)$ is the second order modulus of continuity of the function $f$ defined in \eqref{zr3}.
\end{theorem}

\begin{proof}
We prove this by using the Theorem \eqref{sn2}
\begin{eqnarray*}
\mid D_{n,q}^*(f;x)-f(x)\mid &\leq & \mid D_{n,q}^*(f-g;x)\mid+\mid D_{n,q}^*(g;x)-g(x)\mid+\mid f(x)-g(x)\mid\\
&\leq & 2 \parallel f-g \parallel_{C_B(\mathbb{R}^+)}+(q-1)x\parallel g \parallel_{C_B(\mathbb{R}^+)}+
\frac{\lambda_n(x)}{2}\parallel g \parallel_{C_B^2(\mathbb{R}^+)}\\
\end{eqnarray*}
From \eqref{t1} clearly we have ~~~$\parallel g \parallel_{C_B[0,\infty)}\leq \parallel g \parallel_{C_B^2[0,\infty)}$.\newline
Therefore,
\begin{equation*}
\mid D_{n,q}^*(f;x)-f(x)\mid  \leq  2 \left(\parallel f-g \parallel_{C_B(\mathbb{R}^+)}+\frac{2x(q-1)+ \lambda_n(x)}{4}\parallel g \parallel_{C_B^2(\mathbb{R}^+)}\right),
\end{equation*}
where $\lambda_{n}(x)$ is given in Theorem  \ref{sn1}.\newline

By taking infimum over all $g \in C_B^2(\mathbb{R}^+)$ and by using \eqref{zr1}, we get
\begin{equation*}
\mid D_{n,q}^*(f;x)-f(x)\mid \leq  2K_2\left(f;\frac{2x(q-1)+ \lambda_n(x)}{4}\right)
\end{equation*}
Now for an absolute constant $C>0$ in \cite{c1} we use the relation
 $$K_2(f;\delta)\leq C\{ \omega_2(f;\sqrt{\delta})+\min(1,\delta)\parallel f \parallel\}.$$
 Which complete the proof.
\end{proof}

\section{Construction of bivariate operators }
In this section, we construct a bivariate extension of the operators \eqref{ss1}.

Let $\mathbb{R}^2_+ = [0,\infty)\times [0,\infty),~~f: C(\mathbb{R}^2_+ )\to
\mathbb{R}$ and $0<q_{n_1},q_{n_2}<1$. We define the
bivariate extension of the Dunkl $q$-parametric Sz\'{a}sz-Mirakjan operators \eqref{ss1} as follows:\newline

\begin{equation*}
D_{n_1,n_2}^*(f;q_{n_1},q_{n_2};x,y)=\frac{1}{E_{\mu_1,q_{n_1}}([n_1]_{q_{n_1}}x)}\frac{1}{E_{\mu_2,q_{n_2}}([n_2]_{q_{n_2}}y)}
\sum_{k_1=0}^\infty\sum_{k_2=0}^\infty \frac{([n_1]_{q_{n_1}}x)^{k_1}}{\gamma_{\mu_1,q_{n_1}}(k_1)}
\frac{([n_2]_{q_{n_2}}y)^{k_2}}{\gamma_{\mu_2,q_{n_2}}(k_2)}
\end{equation*}

\begin{equation}\label{st1}
\times q_{n_1}^{\frac{k_1(k_1-1)}{2}}q_{n_2}^{\frac{k_2(k_2-1)}{2}}f \left( \frac{1-q_{n_1}^{2\mu_1\theta_{k_1}+k_1}}{q_{n_1}^{k_1-2}(1-q_{n_1}^{n_1})}, \frac{1-q_{n_2}^{2\mu_2\theta_{k_2}+k_2}}{q_{n_2}^{k_2-2}(1-q_{n_2}^{n_2})}\right)
\end{equation}
where
\begin{equation*}
E_{\mu_1,q_{n_1}}([n_1]_{q_{n_1}}x)%
=\sum_{k_1=0}^\infty\frac{([n_1]_{q_{n_1}}x)^{k_1}}{\gamma_{\mu_1,q_{n_1}}(k_1)}q_{n_1}^{\frac{k_1(k_1-1)}{2}},~~~~~
E_{\mu_2,q_{n_2}}([n_2]_{q_{n_2}}y)%
=\sum_{k_2=0}^\infty\frac{([n_2]_{q_{n_2}}y)^{k_2}}{\gamma_{\mu_2,q_{n_2}}(k_2)}q_{n_2}^{\frac{k_2(k_2-1)}{2}}.
\end{equation*}%

\begin{lemma}
Let $e_{i,j}:\mathbb{R}_{+}^{2}\rightarrow \lbrack 0,\infty)$ such that $e_{i,j}=(uv)^{ij},~~~i,j=0,1,2,\cdots$ be the two
dimensional test functions. Then the $q$-bivariate
operators defined in \eqref{st1} satisfy the following identities:

\begin{enumerate}
\item \label{l1} $%
D_{n_1,n_2}^*(e_{0,0};q_{n_1},q_{n_2};x,y)=1$

\item \label{l2} $%
D_{n_1,n_2}^*(e_{1,0};q_{n_1},q_{n_2};x,y)=q_{n_1}x$

\item \label{l3} $%
D_{n_1,n_2}^*(e_{0,1};q_{n_1},q_{n_2};x,y)=q_{n_2}y$

\item \label{l4} $%
D_{n_1,n_2}^*(e_{2,0};q_{n_1},q_{n_2};x,y)\leq q_{n_1}x^2+\frac{q_{n_1}^{2(1+\mu_1)}}{[n_1]_{q_{n_1}}}[1+2\mu_1]_{q_{n_1}}x$

\item \label{l5}$%
D_{n_1,n_2}^*(e_{0,2};q_{n_1},q_{n_2};x,y)\leq q_{n_2}y^2+\frac{q_{n_2}^{2(1+\mu_2)}}{[n_2]_{q_{n_2}}}[1+2\mu_2]_{q_{n_2}}y$.
\end{enumerate}
\end{lemma}

\begin{lemma}
The $q$-bivariate
operators defined in \eqref{st1} satisfy the following identities.

\begin{enumerate}
\item \label{ll1} $%
D_{n_1,n_2}^*(e_{1,0}-x;q_{n_1},q_{n_2};x,y)=\left(q_{n_1}-1\right)x$

\item \label{ll2} $%
D_{n_1,n_2}^*(e_{0,1}-y;q_{n_1},q_{n_2};x,y)=\left(q_{n_2}-1\right)y$

\item \label{ll3} $%
D_{n_1,n_2}^*\left((e_{1,0}-x\right)^2;q_{n_1},q_{n_2};x,y)\leq \left(1-q_{n_1}\right)x^2+\frac{q_{n_1}^{2(1+\mu_1)}}{[n_1]_{q_{n_1}}}[1+2\mu_1]_{q_{n_1}}x$

\item \label{ll4} $%
D_{n_1,n_2}^*\left((e_{0,1}-y\right)^2;q_{n_1},q_{n_2};x,y)\leq \left(1-q_{n_2}\right)y^2+\frac{q_{n_2}^{2(1+\mu_2)}}{[n_2]_{q_{n_2}}}[1+2\mu_2]_{q_{n_2}}y$.
\end{enumerate}
\end{lemma}

\textbf{\ Rate of Convergence.}\\ \newline

The rate of convergence of operators ${D}%
_{n_1,n_2}^*(f;q_{n_1},q_{n_2};x,y) $ defined in \eqref{st1} by means of modulus
of continuity of some bivariate modulus of smoothness functions are
introduced.\newline

\parindent=8mmIn order to obtain the convergence results for the operators $%
D_{n_1,n_2}^*(f;q_{n_1},q_{n_2};x,y)$, we take $q=q_{n_1},~~q_{n_2}$ where $q_{n_1},q_{n_2}\in (0,1)$ and satisfying,
\begin{align}\label{nasi5}
\lim_{n_1,n_2}q_{n_1}, q_{n_2}\to 1
\end{align}

The modulus of continuity for bivariate case is defined as follows:\newline

For $f\in H_{\omega }(\mathbb{R}_{+}^{2})$\newline
$\widetilde{\omega }(f;\delta _{1},\delta _{2})$
\begin{equation}\label{snn1}
=\sup_{u,x\geq 0}\left\{ %
\bigg{|}f(u,v)-f(x,y)\bigg{|};~~\bigg{|}u-x\bigg{|}%
\leq \delta _{1},\bigg{|}v-y\bigg{|}\leq \delta
_{2},~~(u,v)\in \mathbb{R}_{+}^{2},~~(x,y)\in \mathbb{R}_{+}^{2}\right\}.
\end{equation}

where $H_\omega(\mathbb{R}^+)$ be the space of all real-valued continuous functions $f$.\newline
Then for all $f\in H_{\omega }(\mathbb{R}_{+})$ $\widetilde{\omega }%
(f;\delta _{1},\delta _{2})$ satisfies the following conditions:
\begin{enumerate}
\item[$($i$)$] $\lim_{\delta_1,\delta_2 \to 0}\widetilde{\omega}(f;
\delta_1,\delta_2)\to 0$

\item[$($ii$)$] $\mid f(u,v)-f(x,y) \mid \leq \widetilde{\omega}(f;
\delta_1,\delta_2) \left(\frac{\mid u-x\mid}{\delta_1%
}+1 \right)\left(\frac{\mid v-y\mid}{\delta_2}+1
\right).$
\end{enumerate}

\begin{theorem}\label{tn1}
Let $q_n=q_{n_1},q_{n_2}$ satisfy \eqref{nasi5} and if for $(x,y) \in [0,\infty), ~~0<q_{n_1},q_{n_2}< 1$, and suppose  $%
D_{n_1,n_2}^*(f;q_{n_1},q_{n_2},x,y)$  be the operators defined by \eqref{st1}. Then for
any function $f\in \tilde{C}\left([0,\infty)\times [0,\infty)\right)$, we have \newline
$\mid D_{n_1,n_2}^*(f;q_{n_1},q_{n_2},x,y)-f(x,y) \mid$
\begin{eqnarray*}
 & \leq & \omega\left(f;\frac{1}{\sqrt{[n_1]_{q_{n_1}}}},\frac{1}{\sqrt{[n_2]_{q_{n_2}}}}\right)\left( 1+\sqrt{(1-q_{n_1})[n_1]_{q_{n_1}}x^2+q_{n_1}^{2(1+\mu_1)}[1+2\mu_1]_{q_{n_1}}x}\right)\\
&\times & \left( 1+\sqrt{(1-q_{n_2})[n_2]_{q_{n_2}}y^2+q_{n_2}^{2(1+\mu_2)}[1+2\mu_2]_{q_{n_2}}y}\right)
,
\end{eqnarray*}
where
$\tilde{C}[0,\infty)$ is the space of uniformly continuous functions on $\mathbb{R}^+$ and
$\widetilde{\omega}(f,\delta_{n_1},\delta_{n_2})$ is the modulus of continuity of the function $f \in \tilde{C}\left([0,\infty) \times [0,\infty)\right)$
 defined in \eqref{snn1}.
\end{theorem}

\begin{proof}
We prove it by using the result modulus of continuity and Cauchy-Schwarz inequality:\newline
$\mid D_{n_1,n_2}^*(f;q_{n_1},q_{n_2},x,y)-f(x,y) \mid$
\begin{eqnarray*}
&\leq & \frac{1}{E_{\mu_1,q_{n_1}}([n_1]_{q_{n_1}}x)}\frac{1}{E_{\mu_2,q_{n_2}}([n_2]_{q_{n_2}}y)}\sum_{k_1=0}^\infty \sum_{k_2=0}^\infty \frac{([n_1]_{q_{n_1}}x)^{k_1}}{\gamma_{\mu_1,q_{n_1}}(k_1)}\frac{([n_2]_{q_{n_2}}y)^{k_2}}{\gamma_{\mu_2,q_{n_2}}(k_2)}
q_{n_1}^{\frac{k_1(k_1-1)}{2}}q_{n_2}^{\frac{k_2(k_2-1)}{2}}\\
& \times & \bigg{|}f\left( \frac{1-q_{n_1}^{2\mu_1\theta_{k_1}+k_1}}{q_{n_1}^{k_1-2}(1-q_{n_1}^{n_1})},
\frac{1-q_{n_2}^{2\mu_2\theta_{k_2}+k_2}}{q_{n_2}^{k_2-2}(1-q_{n_2}^{n_2})} \right) -f(x,y) \bigg{|}\\
&\leq & \frac{1}{E_{\mu_1,q_{n_1}}([n_1]_{q_{n_1}}x)}\frac{1}{E_{\mu_2,q_{n_2}}([n_2]_{q_{n_2}}y)}\sum_{k_1=0}^\infty \sum_{k_2=0}^\infty \frac{([n_1]_{q_{n_1}}x)^{k_1}}{\gamma_{\mu_1,q_{n_1}}(k_1)}\frac{([n_2]_{q_{n_2}}y)^{k_2}}{\gamma_{\mu_2,q_{n_2}}(k_2)}
q_{n_1}^{\frac{k_1(k_1-1)}{2}}q_{n_2}^{\frac{k_2(k_2-1)}{2}}\\
& \times &
\left(1+ \frac{1}{\delta_{n_1}} \bigg{|}\left( \frac{1-q_{n_1}^{2\mu_1\theta_{k_1}+k_1}}{q_{n_1}^{k_1-2}(1-q_{n_1}^{n_1})}\right) -x \bigg{|}\right)
\left(1+ \frac{1}{\delta_{n_2}} \bigg{|}\left( \frac{1-q_{n_2}^{2\mu_2\theta_{k_2}+k_2}}{q_{n_2}^{k_2-2}(1-q_{n_2}^{n_2})}\right) -y \bigg{|}\right)\widetilde{\omega}(f;\delta_{n_1},\delta_{n_2})\\
&\leq &\left\{1+\frac{1}{\delta_{n_1}}\left( \frac{1}{E_{\mu_1,q_{n_1}}([n_1]_{q_{n_1}}x)}\sum_{k_1=0}^\infty \frac{([n_1]_{q_{n_1}}x)^{k_1}}{\gamma_{\mu_1,q_{n_1}}(k_1)}q_{n_1}^{\frac{k_1(k_1-1)}{2}}
\left( \frac{1-q_{n_1}^{2\mu_1\theta_{k_1}+k_1}}{q_{n_1}^{k_1-2}(1-q_{n_1}^{n_1})}-x\right)^2\right)^{\frac{1}{2}} \right\}\\
& \times &
\left\{1+\frac{1}{\delta_{n_2}}\left( \frac{1}{E_{\mu_2,q_{n_2}}([n_2]_{q_{n_2}}y)}\sum_{k_2=0}^\infty \frac{([n_2]_{q_{n_2}}y)^{k_2}}{\gamma_{\mu_2,q_{n_2}}(k_2)}q_{n_2}^{\frac{k_2(k_2-1)}{2}}
\left( \frac{1-q_{n_2}^{2\mu_2\theta_{k_2}+k_2}}{q_{n_2}^{k_2-2}(1-q_{n_2}^{n_2})}-y\right)^2\right)^{\frac{1}{2}} \right\}\\
& \times & \widetilde{\omega}(f;\delta_{n_1},\delta_{n_2})\\
&= & \left(1+\frac{1}{\delta_{n_1}} \left(D_{n_1,n_2}^*(e_{1,0}-x)^2;q_{n_1},q_{n_2};x,y\right)^{\frac{1}{2}}\right)
 \left(1+\frac{1}{\delta_{n_2}} \left(D_{n_1,n_2}^*(e_{0,1}-y)^2;q_{n_1},q_{n_2};x,y\right)^{\frac{1}{2}}\right)\\
& \times & \widetilde{\omega}(f;\delta_{n_1},\delta_{n_2})\\
&\leq & \left(1+\frac{1}{\delta_{n_1}} \sqrt{(1-q_{n_1})x^2+\frac{q_{n_1}^{2(1+\mu_1)}}{[n_1]_{q_{n_1}}}[1+2\mu_1]_{q_{n_1}}x}\right)
\\
& \times &
\left(1+\frac{1}{\delta_{n_2}} \sqrt{(1-q_{n_2})x^2+\frac{q_{n_2}^{2(1+\mu_2)}}{[n_2]_{q_{n_2}}}[1+2\mu_2]_{q_{n_2}}y}\right)
\widetilde{\omega}(f;\delta_1,\delta_2).
\end{eqnarray*}
If we choose $\delta_1=\delta_{n_1}=\sqrt{\frac{1}{[n_1]_{q_{n_1}}}}$ and $\delta_2=\delta_{n_2}=\sqrt{\frac{1}{[n_2]_{q_{n_2}}}}$
 then we get our result.
\end{proof}

Now we give the rate of convergence of the operators ${D}%
_{n_1,n_2}^*(f;q_{n_1},q_{n_2};x,y) $ defined in \eqref{st1} in terms of the elements of the usual Lipschitz class $%
Lip_{M}(\nu_1,\nu_2 )$.

Let $f\in C([0,\infty)\times[0,\infty)) $, $M>0$ and $0<\nu_1,\nu_2 \leq 1$. We recall that $f
$ belongs to the class $Lip_{M}(\nu_1,\nu_2 )$ if
\begin{equation}\label{nnn1}
Lip_{M}(\nu_1,\nu_2 )=\left\{ f: \mid f(u,v)-f(x,y)\mid \leq M\mid u-x\mid^{\nu_1 }\mid v-y\mid^{\nu_2 }~~~(u,v)~~~ \mbox{and}~~~(x,y)\in [0,\infty)\right\}
\end{equation}%
is satisfied.\newline

\begin{theorem}\label{sy1}
Let $D_{n_1,n_2}^*(f;q_{n_1},q_{n_2};x,y)$ be the operator defined in \eqref{st1}.
 Then for each $f\in Lip_{M}(\nu_1,\nu_2 ),~~(M>0,~~~0<\nu_1,\nu_2 \leq 1)$ satisfying \eqref{nnn1} we have
\begin{equation*}
\mid D_{n_1,n_2}^*(f;q_{n_1},q_{n_2};x,y)-f(x,y)\mid \leq M \left(\lambda_{n_1}(x)\right)^{\frac{\nu_1}{2}}
\left(\lambda_{n_2}(y)\right)^{\frac{\nu_2}{2}}
\end{equation*}
where $\lambda_{n_1}(x)=D_{n_1,n_2}^*\left((e_{1,0}-x)^2;q_{n_1},q_{n_2};x,y\right),~~~~~
\lambda_{n_2}(y)=D_{n_1,n_2}^*\left((e_{0,1}-y)^2;q_{n_1},q_{n_2};x,y\right)$.
\end{theorem}

\begin{proof}
We prove it by using the result \eqref{nnn1} and H\"{o}lder inequality:\newline
$\mid D_{n_1,n_2}^*(f;q_{n_1},q_{n_2},x,y)-f(x,y) \mid$ \\
\begin{eqnarray*}
 &\leq &\mid D_{n_1,n_2}^*(f(u,v)-f(x,y);q_{n_1},q_{n_2};x,y)\mid\\
 & \leq & D_{n_1,n_2}^*\left(\mid f(u,v)-f(x,y)\mid;q_{n_1},q_{n_2};x,y\right)\\
&\leq & \mid D_{n_1,n_2}^*\left(\mid e_{1,0}-x\mid^{\nu_1};q_{n_1},q_{n_2};x,y\right)
\mid D_{n_1,n_2}^*\left(\mid e_{0,1}-y\mid^{\nu_2};q_{n_1},q_{n_2};x,y\right).
\end{eqnarray*}
Therefore,\newline
$\mid D_{n_1,n_2}^*(f;q_{n_1},q_{n_2};x,y)-f(x,y) \mid $\\
\begin{eqnarray*}
& \leq &  M \frac{1}{E_{\mu_1,q_{n_1}}([n_1]_{q_{n_1}}x)}\sum_{k_1=0}^\infty \frac{([n_1]_{q_{n_1}}x)^{k_1}}{\gamma_{\mu_1,q_{n_1}}(k_1)}q_{n_1}^{\frac{k_1(k_1-1)}{2}}
\bigg{|} \frac{1-q_{n_1}^{2\mu_1\theta_{k_1}+k_1}}{q_{n_1}^{k_1-2}(1-q_{n_1}^{n_1})}-x \bigg{|}^{\nu_1}\\
& \times &   \frac{1}{E_{\mu_2,q_{n_2}}([n_2]_{q_{n_2}}y)}\sum_{k_2=0}^\infty \frac{([n_2]_{q_{n_2}}y)^{k_2}}{\gamma_{\mu_2,q_{n_2}}(k_2)}q_{n_2}^{\frac{k_2(k_2-1)}{2}}
\bigg{|} \frac{1-q_{n_2}^{2\mu_2\theta_{k_2}+k_2}}{q_{n_2}^{k_2-2}(1-q_{n_2}^{n_2})}-y \bigg{|}^{\nu_2}\\
& \leq &  M  \left(\frac{1}{\left(E_{\mu_1,q_{n_1}}([n_1]_{q_{n_1}}x)\right)}\sum_{k_1=0}^\infty \frac{([n_1]_{q_{n_1}}x)^{k_1}q_{n_1}^{\frac{k_1(k_1-1)}{2}}}{\gamma_{\mu_1,q_{n_1}}(k_1)}\right)^{\frac{2-\nu_1}{2}}\\
& \times &
\left(\frac{1}{\left(E_{\mu_1,q_{n_1}}([n_1]_{q_{n_1}}x)\right)}\sum_{k_1=0}^\infty \frac{([n_1]_{q_{n_1}}x)^{k_1}q_{n_1}^{\frac{k_1(k_1-1)}{2}}}{\gamma_{\mu_1,q_{n_1}}(k_1)}\bigg{|} \frac{1-q_{n_1}^{2\mu_1\theta_{k_1}+k_1}}{q_{n_1}^{k_1-2}(1-q_{n_1}^{n_1})}-x \bigg{|}^2\right)^{\frac{\nu_1}{2}}\\
& \times &    \left(\frac{1}{\left(E_{\mu_2,q_{n_2}}([n_2]_{q_{n_2}}y)\right)}\sum_{k_2=0}^\infty \frac{([n_2]_{q_{n_2}}y)^{k_2}q_{n_2}^{\frac{k_2(k_2-1)}{2}}}{\gamma_{\mu_2,q_{n_2}}(k_2)}\right)^{\frac{2-\nu_2}{2}}\\
& \times &
\left(\frac{1}{\left(E_{\mu_2,q_{n_2}}([n_2]_{q_{n_2}}y)\right)}\sum_{k_2=0}^\infty \frac{([n_2]_{q_{n_2}}y)^{k_2}q_{n_2}^{\frac{k_2(k_2-1)}{2}}}{\gamma_{\mu_2,q_{n_2}}(k_2)}\bigg{|} \frac{1-q_{n_2}^{2\mu_2\theta_{k_2}+k_2}}{q_{n_2}^{k_2-2}(1-q_{n_2}^{n_2})}-y \bigg{|}^2\right)^{\frac{\nu_2}{2}}\\
& \leq & \left(D_{n_1,n_2}^*(e_{1,0}-x)^2;q_{n_1},q_{n_2};x,y\right)^{\frac{\nu_1}{2}}
\left(D_{n_1,n_2}^*(e_{0,1}-y)^2;q_{n_1},q_{n_2};x,y\right)^{\frac{\nu_2}{2}}.
\end{eqnarray*}
Which complete the proof.
\end{proof}


{\bf Acknowledgement.} Second author (MN) acknowledges the financial sup-
port of University Grants Commission (Govt. of Ind.) for awarding BSR (Basic
Scientific Research) Fellowship.

\bibliographystyle{amsplain}

\end{document}